\documentclass[12pt]{article}
\usepackage{graphicx} 

\usepackage{authblk}
\usepackage[margin=2.5cm]{geometry}
\usepackage{t1enc}
\usepackage[utf8]{inputenc}
\usepackage{amsmath,amsthm,amssymb}
\usepackage{graphicx}
\usepackage{enumerate}
\usepackage{hyperref}
\usepackage{bm}
\usepackage{comment}
\usepackage{amsfonts}
\usepackage{graphicx,caption}
\usepackage{bm}
\usepackage{amsmath, amsthm, amssymb}
\usepackage{graphicx}
\usepackage{hyperref}
\usepackage{relsize}
\usepackage{blkarray}
\usepackage{tikz}
\usetikzlibrary{decorations.pathreplacing}
\usepackage{tabstackengine}

\stackMath

\usepackage{subfigure}

\usepackage[table]{xcolor}

\usepackage{algpseudocode}

\usepackage{bbm}

\theoremstyle{plain}
\usepackage{amsthm}
\makeatletter
\newcommand{\newreptheorem}[2]{\newtheorem*{rep@#1}{\rep@title}\newenvironment{rep#1}[1]{\def\rep@title{#2 \ref*{##1}}\begin{rep@#1}}{\end{rep@#1}}}
\makeatother

\newtheorem{theorem}{Theorem}
\newtheorem*{theorem-non}{Theorem}
\newtheorem*{non-lemma}{Lemma}
\newtheorem{lemma}[theorem]{Lemma}
\newreptheorem{lemma}{Lemma}

\theoremstyle{definition}

\DeclareMathOperator{\supp}{supp}

\DeclareMathOperator{\suppr}{supp_r}

\title{Cocycles of determinantal hypertrees with small support}
\author{Andr\'as M\'esz\'aros}
\date{}
\affil{HUN-REN Alfr\'ed R\'enyi Institute of Mathematics}

\begin{document}

\maketitle

\begin{abstract}
 Let $\mathcal{T}_n$ be a random $2$-dimensional determinantal hypertree on $n$ vertices. Given any prime $p$, we answer the following question: If a cocycle in $Z^1(\mathcal{T}_n,\mathbb{F}_p)$ has small support, what does the support typically look like? More precisely, we characterize all the finite connected graphs $G$ for which there is a constant $c_G>0$ with the following property: For all large enough $n$, with probability at least $c_G$, we have a cocycle $f\in Z^1(\mathcal{T}_n,\mathbb{F}_p)$ such that after removing all the isolated vertices, the support of $f$ is isomorphic to~$G$. We prove that for $p>2$, we do not have any such graph. For $p=2$, a connected graph has the property above if and only if it has a unique cycle such that this unique cycle has odd length, moreover, if the unique cycle is a triangle, then we also need to require that all the vertices of the triangle have degree at least $3$. 
\end{abstract}

\section{Introduction}

Determinantal hypertrees are natural higher dimensional generalizations of a uniform random spanning tree of a complete graph. They can be defined in any dimension, but in this paper, we restrict our attention to the $2$-dimensional case. A $2$-dimensional simplicial complex $\mathcal{S}$ on the vertex set $[n]=\{1,2,\dots,n\}$ is called a ($2$-dimensional) hypertree, if
\begin{enumerate}[\hspace{30pt}(a)]
 \item\label{pra} $\mathcal{S}$ has complete $1$-skeleton;
 \item\label{prb} The number of triangular faces of $\mathcal{S}$ is ${{n-1}\choose{2}}$;
 \item\label{prc} The homology group $H_{1}(\mathcal{S},\mathbb{Z})$ is finite.
\end{enumerate}

In one dimension, a spanning tree must be connected, property \eqref{prc} above is the two dimensional analogue of this requirement. Note that any complex $\mathcal{S}$ satisfying \eqref{pra} and \eqref{prc} must have at least ${n-1}\choose{2}$ triangular faces. For a graph $G$, if the reduced homology group $\tilde{H}_0(G,\mathbb{Z})$ is finite, then it is trivial. This statement fails in two dimensions since for a hypertree $\mathcal{S}$, the order of $H_1(\mathcal{S},\mathbb{Z})$ can range from $1$ to $\exp(\Theta(n^2))$, see \cite{kalai1983enumeration}. Thus, while the homology of spanning trees is uninteresting, the homology of $2$-dimensional hypertrees is a very rich subject to study.

Kalai's generalization of Cayley's formula \cite{kalai1983enumeration} states that
\[\sum |H_{1}(\mathcal{S},\mathbb{Z})|^2=n^{{n-2}\choose {2}},\]
where the summation is over all the hypertrees $\mathcal{S}$ on the vertex set $[n]$. See also~\cite{duval2009simplicial} for generalizations. This formula suggests that the natural probability measure on the set of hypertrees is the one where the probability assigned to a hypertree $\mathcal{S}$ is \begin{equation}\label{measuredef}
 \frac{|H_{1}(\mathcal{S},\mathbb{Z})|^2}{n^{{n-2}\choose {2}}}.
\end{equation}
It turns out that this measure is a determinantal probability measure \cite{lyons2003determinantal,hough2006determinantal}. Thus, a random hypertree $\mathcal{T}_n$ distributed according to \eqref{measuredef} is called a determinantal hypertree. General random determinantal complexes were investigated by Lyons \cite{lyons2009random}. While uniform random spanning trees are well-studied \cite{ald1,ald2,ald3,grimmett1980random,szekeres2006distribution,lyons2017probability}, a theory of determinantal hypertrees started to emerge only recently. The author determined the local weak limit of determinantal hypertrees~\cite{meszaros2022local}, combining this with estimates on the spectrum of the Laplacian matrix of $\mathcal{T}_n$, the author proved that $n^{-2}\log |H_1(\mathcal{T}_n,\mathbb{Z})|$ converges in probability to a constant~\cite{meszaros2025homology}. Vander Werf~\cite{werf2022determinantal} and the author~\cite{meszaros2023coboundary} investigated various expansion properties of the union of independent copies of $\mathcal{T}_n$. Linial and Peled~\cite{linial2019enumeration} provided estimates on the number of hypertrees. All the above mentioned results extend to dimensions larger than $2$. In the $2$-dimensional case, Kahle and Newman~\cite{kahle2022topology} proved that with high probability the fundamental group $\pi_1(\mathcal{T}_n)$ is hyperbolic and has
cohomological dimension 2. Moreover, the author~\cite{meszaros2024bounds,meszaros2025using} showed that for all prime~$p$, $n^{-2}\dim H_1(\mathcal{T}_n,\mathbb{F}_p)$ converges to $0$ in probability.
 More generally, one can also define spanning hypertrees/hyperforests of a simplicial complex, see for example~\cite[Section 1.2.3]{nachmias2025local}. Note that spanning hypertrees are also called spanning acycles sometimes. The lifetime formula of Hiraoka and Shirai connects minimum weight spanning hypertrees to persistent homology~\cite{hiraoka2017minimum}. Extending the results of the author~\cite{meszaros2022local}, Nachmias and Peled determined the local weak limit of determinantal hyperforests in regular polytopal complexes~\cite{nachmias2025local}.

 Kahle and Newman conjectured that the $p$-part of $H_1(\mathcal{T}_n,\mathbb{Z})$ has Cohen--Lenstra limiting distribution~\cite{kahle2022topology}. In the case $p=2$, the author~\cite{meszaros20242} disproved this conjecture by proving that for any positive integer $h$, provided that $n$ is large enough, we have
 \begin{equation}\label{regibound}
 \mathbb{P}(\dim H_1(\mathcal{T}_n,\mathbb{F}_2)\ge h)\ge \frac{e^{-200h}}{(100h)^{5h}},
 \end{equation}
 showing that $\dim H_1(\mathcal{T}_n,\mathbb{F}_2)$ has a heavier tail than the Cohen--Lenstra limiting distribution would imply. The author established the bound~\eqref{regibound} by proving that with probability $\frac{e^{-200h}}{(100h)^{5h}}$, we have $h$ cocycles in $Z^1(\mathcal{T}_n,\mathbb{F}_2)$ such that each of these cocycles is supported on a pentagon, and these pentagons are pairwise vertex disjoint. Thus, the heavier than Cohen--Lenstra tail of $\dim H_1(\mathcal{T}_n,\mathbb{F}_2)$ can be explained by the presence of cocycles with small support. This seems to be a general phenomenon, see~\cite{meszaros2024phase,meszaros2025rank,kang2026random,lee2025distribution} for other occurrences. Note that the argument above is specific to the case $p=2$, the conjecture of Kahle and Newman is still open for $p>2$. See~\cite{meszaros2023cohen,lee2025distribution} for results inspired by the conjecture above. The Cohen--Lenstra distribution is also conjectured to appear for other random simplicial complexes~\cite{kahle2020cohen}. See also the survey of Wood~\cite{wood2022probability} for instances where the Cohen--Lenstra limiting distribution was established. 

As cocycles with small support played a key role in the discussion above, it is natural to ask for a characterization of all finite graphs that can occur as the support of a cocycle in~$Z^1(\mathcal{T}_n,\mathbb{F}_p)$ with asymptotically positive probability. Theorem~\ref{thm1} and Theorem~\ref{thm2} below provide a full answer to this question, and also show that the case of $p=2$ is drastically different from the case of odd $p$.

Before we state our main theorems, we need a few definitions. We say that a graph $G$ is a nice unicyclic graph (NUG) if
\begin{enumerate}[(a)]
 \item $G$ is connected and contains a unique cycle,
 \item the unique cycle has odd length,
 \item moreover, either
 \begin{itemize}
 \item the unique cycle has length at least $5$, or
 \item the unique cycle is a triangle, and each vertex of this triangle has degree at least $3$.
 \end{itemize}
\end{enumerate}

The support of an $f\in Z^1(\mathcal{T}_n,\mathbb{F}_p)$ is defined as
\[\supp(f)=\left\{\{u,v\}\in {{[n]}\choose{2}}\,:\,f(u,v)\neq 0\right\}, \]
which can be considered as a graph. The graph $\supp_r(f)$ is obtained from $\supp(f)$ by removing all the isolated vertices.

 \begin{theorem}\label{thm1}
There is a $c>0$ such that the following holds: 
\begin{enumerate}[(a)]
 \item\label{thm1parta} As $n\to \infty$, with probability tending to $1$, for all odd primes $p$ and $0\neq f\in Z^1(\mathcal{T}_n,\mathbb{F}_p)$, all the connected components of $\suppr(f)$ have at least $c\log (n)$ vertices. 
 \item\label{thm1partb} As $n\to \infty$, with probability tending to $1$, for all $0\neq f\in Z^1(\mathcal{T}_n,\mathbb{F}_2)$, if a connected component of $\suppr(f)$ has less than $c\log (n)$ vertices, then this connected component must be a NUG. 
\end{enumerate}
\end{theorem}
The next theorem is a counterpart of part~\eqref{thm1partb} of Theorem~\ref{thm1}.
\begin{theorem}\label{thm2}
Let $G$ be a finite graph such that each connected component of $G$ is a NUG. Then for all large enough $n$, we have
\[\mathbb{P}(\,\text{There is an $f\in Z^1(\mathcal{T}_n,\mathbb{F}_2)$ such that $\supp_r(f)\cong G$}\,)\ge \frac{e^{-16Dm}}{(3(m+3))^m},\]
where $m$ is the number of vertices of $G$ and $D$ is the maximum degree of $G$.
\end{theorem}

\textbf{Acknowledgments: } The author was supported by the NKKP-STARTING 150955 project and the Marie Skłodowska-Curie Postdoctoral
Fellowship "RaCoCoLe". The author is grateful to Yuval Peled for the useful discussions.

\section{Preliminaries}\label{secprelim}

Given a simplicial complex $\mathcal{C}$, we use the notation $F_d(\mathcal{C})$ for the set of $d$-dimensional faces of~$\mathcal{C}$. Moreover, let $H_{\mathcal{C}}$ be the bipartite graph with color classes $F_1(\mathcal{C})$ and $F_2(\mathcal{C})$ such that the edge between a $\tau\in F_1(\mathcal{C})$ and a $\sigma\in F_2(\mathcal{C})$ is present if and only if $\tau\subset \sigma$.

Let $M$ be a matrix. For a subset $A$ of the rows of $M$ and a subset $B$ of the columns of $M$, the corresponding submatrix of $M$ will be denoted by $M[A,B]$. If $B$ is the set of all columns, we use the notation $M[A,*]$. If $A$ is the set of all rows, we use the notation $M[*,B]$ or just simply $M[B]$.

For $d\ge 1$, let $J_{n,d}$ be a matrix indexed by ${{[n]}\choose {d}}\times {{[n]}\choose{d+1}}$ defined as follows. Let $\sigma=\{x_0,x_1,\dots,x_d\}\subset [n]$ such that $x_0<x_1<\dots<x_d$. For a $\tau\in {{[n]}\choose {d}}$, we set
\[J_{n,d}(\tau,\sigma)=\begin{cases}
(-1)^i&\text{if }\tau=\sigma\setminus\{x_i\},\\
0&\text{otherwise.}
\end{cases}
\]
Note that $J_{n,d}$ is just the matrix of the $d$th boundary map of the simplex on $[n]$.

Let
\[J_{n,d}^r=J_{n,d}\left[{{[n-1]}\choose{d}},*\right].\]

The next lemma was proved in \cite[Lemma 2]{kalai1983enumeration}.
\begin{lemma}\label{Hdet}
Let $\mathcal{C}$ be a $2$-dimensional simplicial complex on the vertex set $[n]$ with complete $1$-skeleton and ${n-1}\choose 2$ triangular faces. Then $\mathcal{C}$ is a hypertree if and only if $\det J^r_{n,2}[F_2(\mathcal{C})]\neq 0$. Moreover, if $\mathcal{C}$ is a hypertree, then 
\[|H_1(\mathcal{C},\mathbb{Z})|=|\det J^r_{n,2}[F_2(\mathcal{C})]|.\]
\end{lemma}

We will also rely on the following estimate: For any $F\subset {{[n]}\choose{3}}$, we have
\begin{equation}\label{hadamard}\mathbb{P}(F\subset \mathcal{T}_n)\le \left(\frac{3}n\right)^{|F|},\end{equation}
see for example \cite[Section 2]{kahle2022topology}.

\section{Non-existence results -- The proof of Theorem~\ref{thm1}}

\subsection{Deterministic results}

We prove part~\eqref{thm1parta} of Theorem~\ref{thm1} in two steps. First, in this section, we prove that if for a $p>2$ and $0\neq f\in Z^1(\mathcal{T}_n,\mathbb{F}_p)$, $\suppr(f)$ has a connected component with less than $c\log (n)$ vertices, then $\mathcal{T}_n$ must contain a small subcomplex which is dense in a certain sense, see Lemma~\ref{oddpdense}. As a second step, in Section~\ref{SecDense}, we use the first moment method to show that with probability tending to $1$, $\mathcal{T}_n$ does not have any small dense subcomplexes. The proof of part~\eqref{thm1partb} of Theorem~\ref{thm1} is very similar.

Note that $H_1(\mathcal{T}_n,\mathbb{R})\cong H^1(\mathcal{T}_n,\mathbb{R})$ is always trivial. In this section, we rely on the following idea a lot: We can prove that $Z^1(\mathcal{T}_n,\mathbb{F}_p)$ can not contain certain types of cocycles $f$, by showing that we can build a cocycle $f'\in Z^1(\mathcal{T}_n,\mathbb{R})\setminus B^1(\mathcal{T}_n,\mathbb{R})$ out of $f$, and thus contradicting the triviality of $H^1(\mathcal{T}_n,\mathbb{R})$.

\begin{lemma}\label{bipartitegen}
Let $\mathcal{K}$ be a $2$-dimensional simplicial complex, and $ f\in Z^1(\mathcal{K},\mathbb{F}_p)$. Assume that there is a partition of the vertex set of $\mathcal{K}$ into three sets $A,B$ and $C$ with the following properties: 
\begin{itemize}
 \item No edge of $\supp(f)$ connects two vertices which belong to the same part of the partition $A\cup B\cup C$.
 \item If $\{a,b,c\}$ is a triangular face of $\mathcal{K}$ such that $a\in A$, $b\in B$ and $c\in C$, then it cannot happen that both $\{a,c\}$ and $\{b,c\}$ are contained in $\supp(f)$. 
\end{itemize}
Then there is a $f'\in Z^1(\mathcal{K},\mathbb{R})$ such that $\supp(f')= \supp(f)$.
\end{lemma}

\begin{proof}
 We define $f'\in C^1(\mathcal{K},\mathbb{R})$ by the formula
\[f'(u,v)=\begin{cases}
+1&\text{ if $u\in A, v\in B\cup C$ and $\{u,v\}\in\supp(f)$},\\
-1&\text{ if $u\in B, v\in A\cup C$ and $\{u,v\}\in\supp(f)$},\\
-1&\text{ if $u\in C, v\in A$ and $\{u,v\}\in\supp(f)$},\\
+1&\text{ if $u\in C, v\in B$ and $\{u,v\}\in\supp(f)$},\\
0&\text{ if $\{u,v\}\notin \supp(f)$.}
\end{cases}\]
It is straightforward to check that $f'(u,v)=-f'(v,u)$ for all $\{u,v\}\in F_1(\mathcal{K})$, so $f'$ is indeed an element of $C^1(\mathcal{K},\mathbb{R})$. It is also clear that $\supp(f')=\supp(f)$. 

To prove that $f'\in Z^1(\mathcal{K},\mathbb{R})$, we need to prove that
\begin{equation}\label{kell3}
 f'(u,v)+f'(v,w)+f'(w,u)=0
\end{equation}
for each triangular face $\{u,v,w\}$ of $\mathcal{K}$. If $u,v,w$ are all contained in one part of the partition $A\cup B\cup C$, then \eqref{kell3} is clear, because $f'(u,v),f'(u,w)$ and $f'(w,u)$ are all equal to $0$. Now consider the case when two vertices of the triangle are in one part, while the third one is in another one. Assume that $u\in A$ and $v,w\in B$. The other cases are similar. Then $f'(v,w)=0$. If neither of $\{u,v\}$ and $\{u,w\}$ is in $\supp(f)$, then $f'(u,v)=f'(w,u)=0$, so \eqref{kell3} holds. If both $\{u,v\}$ and $\{u,w\}$ are in $\supp(f)$, then
\[f'(u,v)+f'(v,w)+f'(w,u)=1+0-1=0,\]
so \eqref{kell3} is again true. It is impossible to have $\{u,v\}\in \supp(f)$ and $\{u,w\}\notin\supp(f)$, because in that case
\[f(u,v)+f(v,w)+f(w,u)=f(u,v)\neq 0,\]
which contradicts the fact that $f\in Z^1(\mathcal{K},\mathbb{F}_p)$. Similarly, it is also impossible to have \break$\{u,v\}\notin \supp(f)$ and $\{u,w\}\in\supp(f)$.

Finally assume that $u\in A$, $v\in B$ and $w\in C$. By our assumption $\{u,w\}\notin \supp(f)$ or $\{v,w\}\notin \supp(f)$. By symmetry, we may assume that $\{v,w\}\notin \supp(f)$, so $f'(v,w)=0$. As before, since $f\in Z^1(\mathcal{K},\mathbb{F}_p)$, it can not happen that exactly one of $\{u,v\}$ and $\{u,w\}$ is in~$\supp(f)$. If neither of $\{u,v\}$ and $\{u,w\}$ is in $\supp(f)$, then $f'(u,v)=f'(w,u)=0$, so \eqref{kell3} holds. If both $\{u,v\}$ and $\{u,w\}$ are in $\supp(f)$, then
$f'(u,v)+f'(v,w)+f'(w,u)=1+0-1=0$
so \eqref{kell3} holds.
\end{proof}

\begin{lemma}\label{lemmabipartite}
 Let $\mathcal{K}$ be a $2$-dimensional simplicial complex, and $ f\in Z^1(\mathcal{K},\mathbb{F}_p)$. Assume that the graph $\supp(f)$ is bipartite. Then there is a $f'\in Z^1(\mathcal{K},\mathbb{R})$ such that $\supp(f')= \supp(f)$. 
\end{lemma}
\begin{proof}
Let $A$ and $B$ be the two color classes of $\supp(f)$, and let $C=\emptyset$. The statement follows from Lemma~\ref{bipartitegen}.\end{proof}

\begin{lemma}\label{trianglesinsup}
 Let $p$ be an odd prime. Let $\mathcal{K}$ be a $2$-dimensional simplicial complex, and let $f\in Z^1(\mathcal{K},\mathbb{F}_p)$. Assume that if the graph $\supp(f)$ contains a triangle, then this triangle is not a $2$-dimensional face of $\mathcal{K}$. Then there is an $f'\in Z^1(\mathcal{K},\mathbb{R})$ such that $\supp(f')= \supp(f)$.
\end{lemma}
\begin{proof}
 We define the map $\ell_p:\mathbb{F}_p\to\mathbb{Z}$ as follows. Given $x\in \mathbb{F}_p=\mathbb{Z}/p\mathbb{Z}$, we define $\ell_p(x)$ as the unique $y\in \{-\frac{p-1}2,-\frac{p-1}2+1,\dots,\frac{p-1}2-1,\frac{p-1}2\}$ such that $y+p\mathbb{Z}=x$. Note that
\begin{equation}\label{basicliftproperty0}\ell_p(x)=0\text{ if and only if }x=0,\end{equation}
\begin{equation}\label{basicliftproperty}\ell_p(x)+\ell_p(-x)=0\text{ for all }x\in \mathbb{F}_p.\end{equation}

 Let us define $f'\in C^1(\mathcal{K},\mathbb{R})$ by the formula $f'(u,v)=\ell_p(f(u,v))$.
 By \eqref{basicliftproperty}, we have $f'(u,v)=-f'(v,u)$ for all $\{u,v\}\in F_1(\mathcal{K})$, so this definition indeed gives an element of~$C^1(\mathcal{K},\mathbb{R})$. 

 By \eqref{basicliftproperty0}, we see that $\supp(f')=\supp(f)$, it remains to show that $f'\in Z^1(\mathcal{K},\mathbb{R})$. That is, given any triangular face $\{u,v,w\}$ of $\mathcal{K}$, we need to show that 
 \begin{equation}\label{kell}
 f'(u,v)+f'(v,w)+f'(w,u)=0.
 \end{equation}
 Note that since $f\in Z^1(\mathcal{K},\mathbb{F}_p)$, we have
 \begin{equation}\label{tudjuk}
 f(u,v)+f(v,w)+f(w,u)=0.
 \end{equation}
 By our assumptions on $\supp(f)$, at least one of $f(u,v), f(v,w), f(w,u)$ equals to $0$. We may assume that $f(u,v)=0$. Combining this with \eqref{tudjuk}, we see that $f(v,w)+f(w,u)=0$. By~\eqref{basicliftproperty}, this implies that $\ell_p(f(v,w))+\ell_p(f(w,u))=0$. Comparing these observations with the definition of $f'$, we see that $f'(u,v)=0$ and $f'(v,w)+f'(w,u)=0$. Thus, \eqref{kell} indeed holds. 
\end{proof}

\begin{lemma}\label{lemmanontrivialsol}
Let $\mathcal{K}$ be a $2$-dimensional simplicial complex, and let $F\subset F_1(\mathcal{K})$ such that $|F|>|F_2(\mathcal{K})|$. Then there is a non-zero $f\in Z^1(\mathcal{K},\mathbb{R})$ such that $\supp(f)\subseteq F$.
\end{lemma}
\begin{proof}
Fix a total ordering of the vertices of $\mathcal{K}$, and let us define
\[I=\{(u,v)\,:\,\{u,v\}\in F, u<v\},\]
moreover, given $\sigma=\{u,v,w\}\in F_2(\mathcal{K})$ such that $u<v<w$, we define 
\[\beta_{(u,v),\sigma}=1,\quad\beta_{(u,w),\sigma}=-1,\quad \beta_{(v,w),\sigma}=1. \]

Consider the following linear system of equations with variables $\left(f(u,v)\right)_{(u,v)\in I}$:
\[\sum_{\substack{(u,v)\in I\\\{u,v\}\subset \sigma}} \beta_{(u,v),\sigma} f(u,v)=0\text{ for all }\sigma\in F_2(\mathcal{K}).\]
Since we have more variables than equations, we have a non-trivial solution, which gives the desired cocycle.
\end{proof}

Let $\mathcal{C}$ be a $2$-dimensional simplicial complex. Given $S\subset F_1(\mathcal{C})$, the neighborhood of $S$ is a subcomplex $\mathcal{N}(\mathcal{C},S)$ of $\mathcal{C}$ defined as follows:
\begin{align*}F_2(\mathcal{N}(\mathcal{C},S))&=\{\sigma\in F_2(\mathcal{C})\,:\, \tau\subset\sigma\text{ for some }\tau\in S\},\\
F_1(\mathcal{N}(\mathcal{C},S))&=S\cup \{\tau\,:\, \tau\text{ is a $2$-element subset of some $\sigma\in F_2(\mathcal{N}(\mathcal{C},S))$}\},\\
F_0(\mathcal{N}(\mathcal{C},S))&=\bigcup_{\tau\in F_1(\mathcal{N}(\mathcal{C},S))} \tau.
\end{align*}

\begin{lemma}\label{lemmaextend}
 Let $\mathbb{F}$ be any field. Let $\mathcal{C}$ be a $2$-dimensional simplicial complex and $S\subset F_1(\mathcal{C})$. Given $f\in C^1(\mathcal{N}(\mathcal{C},S),\mathbb{F})$, let $f'\in C^1(\mathcal{C},\mathbb{F})$ be obtained from $f$ by extending it with zeros. Assuming that $\supp(f)\subset S$, we have $f\in Z^1(\mathcal{N}(\mathcal{C},S),\mathbb{F})$ if and only if $f'\in Z^1(\mathcal{C},\mathbb{F})$.
\end{lemma}
\begin{proof}
 It is straightforward to see that if $f'\in Z^1(\mathcal{C},\mathbb{F})$, then $f\in Z^1(\mathcal{N}(\mathcal{C},S),\mathbb{F})$. Now assume that $f\in Z^1(\mathcal{N}(\mathcal{C},S),\mathbb{F})$. We need to prove that for any triangle, $\sigma=\{u,v,w\}\in\mathcal{C}$, we have \begin{equation}\label{kell5}f'(u,v)+f'(v,w)+f'(w,u)=0.\end{equation} 
 If $\sigma\in \mathcal{N}(\mathcal{C},S)$, then $f'(a,b)=f(a,b)$ for every edge $\{a,b\}\subset \sigma$. Thus, \eqref{kell5} is true, since $f\in Z^1(\mathcal{N}(\mathcal{C},S),\mathbb{F})$. If $\sigma\notin \mathcal{N}(\mathcal{C},S)$, then the boundary of $\sigma$ can not contain any edge from~$S$, so $f'(a,b)=0$ for all edge $\{a,b\}\subset \sigma$. Thus, \eqref{kell5} is true again.
\end{proof}

Recall the definition of the bipartite graph $H_{\mathcal{C}}$ from the beginning of Section~\ref{secprelim}.
\begin{lemma}\label{lemmaconnectedcomponents}
Let $\mathcal{K}$ be a $2$-dimensional simplicial complex, and $f\in Z^1(\mathcal{K},\mathbb{F}_p)$. Let $S$ be the set of the edges of a connected component of $\suppr(f)$. Then there is a partition $(S_i)_{i=1}^r$ of $S$ with the following properties: 
\begin{itemize}
\item Let $\mathcal{K}_i=\mathcal{N}(\mathcal{K},S_i)$. Then the sets $F_2(\mathcal{K}_1),F_2(\mathcal{K}_2),\dots,F_2(\mathcal{K}_r)$ are pairwise disjoint.

\item For each $i$, the edges of $S_i$ form a connected graph. With some abuse of notation we will also denote this graph by $S_i$. 

\item For each $i$, $F_1(\mathcal{K}_i)\cap \supp(f)=S_i$.

\item For each $i$, the vertex set of $\mathcal{K}_i$ is the same as the vertex set of $S_i$.

\item For each $i$, the bipartite graph $H_{\mathcal{K}_i}$ is connected.

\item For each $i$, the subgraph of $H_{\mathcal{K}_i}$ induced by $S_i\cup F_2(\mathcal{K}_i)$ is connected.


\end{itemize}

\end{lemma}
\begin{proof}
 Let $H$ be the subgraph $H_{\mathcal{K}}$ induced by $S\cup F_2(\mathcal{K})$. Let $C_1,C_2,\dots,C_r$ be the connected components of $H$ containing at least one vertex from $S$. Let $S_i=S\cap V(C_i)$. Clearly, $(S_i)_{i=1}^r$ gives a partition of $S$. Note that $F_2(\mathcal{K}_i)=V(C_i)\cap F_2(\mathcal{K})$. Therefore, the sets $F_2(\mathcal{K}_1),F_2(\mathcal{K}_2),\dots,F_2(\mathcal{K}_r)$ are pairwise disjoint. Moreover, it also follows that the subgraph of $H_{\mathcal{K}_i}$ induced by $S_i\cup F_2(\mathcal{K}_i)$ is just $C_i$, so it is connected. The graph $H_{\mathcal{K}_i}$ is obtained from $C_i$ by adding the vertices in $F_1(\mathcal{K}_i)\setminus S_i$. Since each of these new vertices are connected to a vertex in $C_i$, we see that $H_{\mathcal{K}_i}$ is connected. Take any triangular face $\sigma$ of $\mathcal{K}_i$. By definition, $\sigma$ is connected to a vertex in the graph $H_{\mathcal{K}}$ such that this vertex corresponds to an edge $\tau\in S_i$. Since $f$ is a cocycle, it can not happen that in the graph $H_{\mathcal{K}}$, $\tau$ is the only neighbor of $\sigma$  which is in $\supp(f)$. Let $\tau'\neq\tau$ be a neighbor of $\sigma$ which is in $\supp(f)$. Since $\tau$ and $\tau'$ are adjacent edges of the graph $\supp(f)$, we must have $\tau'\in S$. But then $\tau'$ is in the connected component $C_i$, which implies that $\tau'\in S_i$. It follows that $F_1(\mathcal{K}_i)\cap \supp(f)=S_i$. Since $\sigma=\tau\cup \tau'$, it follows that $\sigma$ contained in the vertex set of~$S_i$. Thus, it follows easily that $S_i$ and $\mathcal{K}_i$ have the same vertex set. Also if $\tau_1,\tau_2$ has a common neighbor in the graph $H$, then the edges $\tau_1$ are $\tau_2$ are adjacent. Combining this with the fact that $C_i$ is connected, it follows that $S_i$ is a connected graph. \end{proof}

\begin{lemma}\label{elementarygraph}
 Let $G$ be a connected graph on at least 3 vertices, and let $U$ be a $3$-element subset of the vertices. Then there is a vertex $a\in U$ with the following property: Let $\{b,c\}=U\setminus \{a\}$, then there is a path between $b$ and $c$ that avoids $a$. 
\end{lemma}
\begin{proof}
 If for all pair of vertices in $U$ there is a path between them that avoids the third, then the statement is clear, since we can choose any element of $U$ as $a$. Thus, assume that there are two vertices $a$ and $b$ such that all the paths between $a$ and $b$ goes through $c$. Since $G$ is connected, we have at least one path $P$ from $a$ to $b$ goes through $c$. Consider the segment of $P$ that goes from $b$ to $c$, this path shows that $a$ satisfies the requirements of the lemma.
\end{proof}


\begin{lemma}\label{ordense}
Let $\mathcal{K}$ be a $2$-dimensional simplicial complex, and $f\in Z^1(\mathcal{K},\mathbb{F}_p)$. Assume that $\supp(f)$ is a (non-empty) connected graph $S$ on $k$ vertices. Then at least one of the following is true
\begin{enumerate}[(a)]
 \item\label{case1} There is non-zero $f'\in Z^1(\mathcal{K},\mathbb{R})$ such that $\supp(f')\subseteq \supp(f)$.
 \item\label{case2} $|F_2(\mathcal{K})|\ge k+1$.
 \item\label{case3} $p=2$, $S$ is a NUG and $|F_2(\mathcal{K})|\ge k$.
\end{enumerate}
\end{lemma}
\begin{proof}

For the sake of contradiction, assume that \eqref{case1}, \eqref{case2} and \eqref{case3} all fail.

By Lemma~\ref{lemmanontrivialsol}, since \eqref{case1} does not hold, $|\supp(f)|\le |F_2(\mathcal{K})|$. If \eqref{case2} also fails, this implies that $|\supp (f)|\le k$. Combining this with the fact that $S$ is a connected graph on $k$ vertices, we see that $\supp(f)$ contains at most one cycle. Moreover, $\supp(f)$ must contain an odd cycle, because otherwise $\supp(f)$ would be bipartite, so we could use Lemma~\ref{lemmabipartite} to show that \eqref{case1} holds. Thus, $\supp(f)$ must be an unicyclic graph, where the unique cycle has odd length. Using our earlier observations, we have $|F_2(\mathcal{K})|\ge |\supp(f)|=k$. We claim that the unique cycle must be a triangle $\sigma$. If $p=2$, then it follows from the fact that \eqref{case3} also fails. In the case of odd $p$, if the length of the unique cycle was at least $5$, then $\supp(f)$ would be triangle free, so \eqref{case1} would occur by Lemma~\ref{trianglesinsup}.

First assume that $\sigma\in \mathcal{K}$, where $\sigma$ the unique triangle in $S$. 

Let $\mathcal{K}'=\mathcal{K}\setminus\{\sigma\}$. Obviously, $f\in Z^1(\mathcal{K}',\mathbb{F}_p)$. Let us apply Lemma~\ref{lemmaconnectedcomponents} to the complex~$\mathcal{K}'$. Let $S=(S_i)_{i=1}^r$ be the resulting partition of $S$, and let $\mathcal{K}_i=\mathcal{N}(\mathcal{K}',S_i)$. Let $f_i$ be the restriction of $f$ to $F_1(\mathcal{K}_i)$. Clearly, $f_i\in Z^1(\mathcal{K}_i,\mathbb{F}_p)$ and $\supp(f_i)=S_i$.

 First, assume that there is an $i$ such that the intersection of $S_i$ and the boundary of $\sigma$ contains a single edge $\{u,v\}$. Let $w$ the third vertex of the triangle $\sigma$. In this case, $r$ can not be $1$. Thus, there is a $j\in\{1,\dots,r\}$ such that $i\neq j$. It is clear that $S_j=\suppr(f_j)$ is a tree. In particular, $\supp(f_j)$ is bipartite. So by combining Lemma~\ref{lemmabipartite} and Lemma~\ref{lemmaextend}, we can find a $f_j'\in Z^1(\mathcal{K}',\mathbb{R})$ such that $\supp(f_j')= \supp(f_j)$. Let $a=f_j'(v,w)+f_j'(w,u)$. Since $\supp(f_i)$ is also bipartite, by combining Lemma~\ref{lemmabipartite} and Lemma~\ref{lemmaextend}, we can find a $f_i'\in Z^1(\mathcal{K}',\mathbb{R})$ such that $\supp(f_i')= \supp(f_i)$. By multiplying $f_i'$ with a scalar, we can get an $f_i''\in Z^1(\mathcal{K}',\mathbb{R})$ such that $\supp(f_i'')\subset \supp(f_i)$ and $f_i''(u,v)=a$. We set $f'=f_j'-f_i''$. Then $f'$ is a nonzero element $Z^1(\mathcal{K}',\mathbb{R})$. Moreover,
\[f'(u,v)+f'(v,w)+f'(w,u)=-f_i''(u,v)+f_j'(v,w)+f_j'(w,u)=-a+a=0.\]
Thus, $f'$ is in fact a non-zero element of $Z^1(\mathcal{K},\mathbb{R})$. Thus, \eqref{case1} holds, which is a contradiction.

Therefore the only option is that all the edges of $\sigma=\{u,v,w\}$ are contained in some $S_i$. By Lemma~\ref{lemmaconnectedcomponents}, the subgraph $H$ of $H_{\mathcal{K}_i}$ induced by $S_i\cup F_2(\mathcal{K}_i)$ is connected. By Lemma~\ref{elementarygraph}, without the loss of generality, we may assume that there is a path in $H$, from $\{u,v\}$ to $\{v,w\}$ that avoids $\{w,u\}$.

Consider the graph $S'$ obtained from $S$ by deleting the edge $\{w,u\}$. Then $S'$ is bipartite. Let $A$ and $B$ be the two color classes such that $u\in A$ and $v\in B$. Clearly we must have $w\in A$. Since there is a path in $H$, from $\{u,v\}$ to $\{v,w\}$ that avoids $\{w,u\}$, there is a sequence $(u,v)=(a_0,b_0),(a_1,b_1),\dots,(a_\ell,b_\ell)=(w,v)$ of oriented edges of $S'$ such that for all~$i$, $\{a_i,b_i\}$ and $\{a_{i+1},b_{i+1}\}$ have a common neighbor in $H$. Here we choose the orientation of the edges such that $a_i\in A$ and $b_i\in B$ for all $i$. Let $0\le i<\ell$. The fact that $\{a_i,b_i\}$ and $\{a_{i+1},b_{i+1}\}$ have a common neighbor in $H$ implies that either
\begin{itemize}
 \item $a_i=a_{i+1}$, $b_i\neq b_{i+1}$ and $\{a_i,b_i,b_{i+1}\}\in\mathcal{K}'$, or
 \item $a_i\neq a_{i+1}$, $b_i= b_{i+1}$ and $\{a_i,a_{i+1},b_{i}\}\in\mathcal{K}'$.
\end{itemize}

In the first case, since $f$ is a cocycle, we have
\[0=f(a_i,b_i)+f(b_i,b_{i+1})+f(b_{i+1},a_i)=f(a_i,b_i)+0-f(a_{i+1},b_{i+1}).\]
Note that here $f(b_i,b_{i+1})=0$ because the only triangle in the graph $S$ is $\sigma=\{u,v,w\}$ and $\sigma\notin\mathcal{K}'$, so since $\{a_i,b_i\}\in S$ and $\{b_{i+1},a_i\}\in S$, we must have $\{b_i,b_{i+1}\}\notin S$. 
Thus, $f(a_{i+1},b_{i+1})=f(a_i,b_i)$. By a similar proof this is also true in the second case. Therefore, it follows that $f(u,v)=f(w,v)$.

 Thus,
\[f(u,v)+f(v,w)+f(w,u)=f(w,u)\neq 0,\]
which contradicts the facts that $\{u,v,w\}\in\mathcal{K}$ and $f\in Z^1(\mathcal{K},\mathbb{F}_p)$.

Finally, consider the case when $\sigma\notin\mathcal{K}$. If $p$ is odd, then Lemma~\ref{trianglesinsup} shows that \eqref{case1} holds, which is a contradiction. If $p=2$, since \eqref{case3} fails, we have a vertex $v$ of the triangle $\sigma$ such that the degree of $v$ in the graph $\supp(f)$ is $2$. If we delete the vertex $v$ from $\supp(f)$, we get a bipartite graph. Let $A$ and $B$ the two color classes of this graph, and let $C=\{v\}$. Then Lemma~\ref{bipartitegen} can be applied to show that \eqref{case1} holds, which is a contradiction. 
\end{proof}

We say that a 2-dimensional simplicial complex $\mathcal{C}$ is $k$-dense, if it has $k$ vertices and at least $k+1$ triangular faces and $H_{\mathcal{C}}$ is connected.

We say that a 2-dimensional simplicial complex $\mathcal{C}$ is $k$-semidense, if $\mathcal{C}$ has $k$ vertices, and for some $r>1$, $\mathcal{C}$ can be obtained as $\mathcal{C}=\cup_{i=1}^r \mathcal{C}_i$ for some simplicial complexes $\mathcal{C}_i$ such that  $H_{\mathcal{C}_i}$ is connected, $F_2(\mathcal{C}_1),F_2(\mathcal{C}_2),\dots,F_2(\mathcal{C}_r)$ are pairwise disjoint, moreover, $\mathcal{C}$ has at least $k+r-1$ triangular faces.

\begin{lemma}\label{cobwithiso}
 Let $K_n$ be complete graph on the vertex set $V$, and let $g\in B^1(K_n,\mathbb{R})$ such that the graph $\supp(g)$ contains an isolated vertex $v$. Then $g=0$.
\end{lemma}
\begin{proof}
 Let $g$ be the coboundary of a function $f:[n]\to \mathbb{R}$. Since $g(v,u)=0$ for all $u\neq v$, it follows that $f(u)=f(v)$ for all $u\neq v$. Thus, $f$ is constant, and therefore $g=0$. 
\end{proof}

\begin{lemma}\label{oddpdense}
Let $p$ be an odd prime. Let $\mathcal{T}_n$ be a hypertree on $n$ vertices, and let $f\in Z^1(\mathcal{T}_n,\mathbb{F}_p)$ such that $\supp(f)$ has a connected component $S$ with $m$ vertices, where $1<m<n$. Then $\mathcal{T}_n$ has a $k$-dense subcomplex for some $k\le m$.\end{lemma}
\begin{proof}
 Applying Lemma~\ref{lemmaconnectedcomponents} with the choice of $\mathcal{K}=\mathcal{T}_n$, we obtain a partition $(S_i)_{i=1}^r$ of $S$. Let $\mathcal{K}_1=\mathcal{N}(\mathcal{T}_n,S_1)$, and let $f_1\in Z^1(\mathcal{K}_1,\mathbb{F}_p)$ be the restriction of $f$ to $F_1(\mathcal{K}_1)$. 
By Lemma~\ref{lemmaconnectedcomponents}, $H_{\mathcal{K}_1}$ is connected and $S_1=\suppr(f_1)$. Let $k$ be the size of the vertex set of $S_1$. By Lemma~\ref{ordense} one of the following is true
 \begin{itemize}
 \item There is non-zero $f'\in Z^1(\mathcal{K}_1,\mathbb{R})$ such that $\supp(f_1')\subseteq \supp(f_1)$.
 \item $|F_2(\mathcal{K}_1)|\ge k+1$.
 
 \end{itemize}

 If the second statement is true, the $\mathcal{K}_1$ is a $k$-dense complex. Thus, it is enough to exclude the possibility that the first statement is true. If there was a non-zero $f_1'\in Z^1(\mathcal{K}_1,\mathbb{R})$ such that $\supp(f_1')\subseteq \supp(f_1)$, then Lemma~\ref{lemmaextend} would give us a non-zero $g\in Z^1(\mathcal{T}_n,\mathbb{R})$. Since $H^1(\mathcal{T}_n,\mathbb{R})\cong H_1(\mathcal{T}_n,\mathbb{R})$ is trivial, this would imply that $g\in B^1(\mathcal{T}_n,\mathbb{R})$. Since we assumed that $m<n$, this would contradict Lemma~\ref{cobwithiso}. 
 \end{proof}

\begin{lemma}\label{p2dense}
 Let $\mathcal{T}_n$ be a hypertree on $n$ vertices, and let $f\in Z^1(\mathcal{T}_n,\mathbb{F}_2)$. Assume that $\supp(f)$ has a connected component $S$ with $m$ vertices such that $S$ is not a NUG and $1<m<n$. Then $\mathcal{T}_n$ has a $k$-dense or a $k$-semidense subcomplex for some $k\le m$.\end{lemma}
\begin{proof}
Applying Lemma~\ref{lemmaconnectedcomponents} with the choice of $\mathcal{K}=\mathcal{T}_n$, we obtain a partition $(S_i)_{i=1}^{r}$ of $S$. Let $\mathcal{K}_i=\mathcal{N}(\mathcal{T}_n,S_i)$, and let $f_i\in Z^1(\mathcal{K}_i,\mathbb{F}_2)$ be the restriction of $f$ to $F_1(\mathcal{K}_i)$.

By Lemma~\ref{lemmaconnectedcomponents}, $H_{\mathcal{K}_i}$ is connected, and $S_i=\suppr(f_i)$ is a connected graph on $ k_i$ vertices, where $k_i\le m$. Moreover, $F_2(\mathcal{K}_1),\dots,F_2(\mathcal{K}_r)$ are pairwise disjoint. By Lemma~\ref{ordense}, given an~$i$, one of the following is true
 \begin{itemize}
 \item There is non-zero $f'\in Z^1(\mathcal{K}_i,\mathbb{R})$ such that $\supp(f_i')\subseteq \supp(f_i)$.
 \item $|F_2(\mathcal{K}_i)|\ge k_i+1$.
 \item $S_i$ is a NUG and $|F_2(\mathcal{K}_i)|\ge k_i$.
 \end{itemize}

 As we have seen in the proof of Lemma~\ref{oddpdense}, the first statement can not be true. If the second statement is true for some $i$, then we have found a $k_i$-dense subcomplex of $\mathcal{T}_n$. Thus, it is enough to consider the case when for all $i$, $S_i$ is a NUG and $|F_2(\mathcal{K}_i)|\ge k_i$. Since we assumed that $S$ is not a NUG, we see that $r>1$. Therefore, $S$ is a connected graph on $m$ vertices such that it contains $r$ pairwise disjoint cycles. Thus, on the one hand, the number of edges of $S$ is at least $m+r-1$. On the other hand, the number of edges of $S$ is $\sum_{i=1}^{r} |S_i|=\sum_{i=1}^r k_i$. Observe that the number of triangular faces of $\mathcal{N}(\mathcal{T}_n,S)=\cup_{i=1}^r \mathcal{K}_i$ is $\sum_{i=1}^r |F_2(\mathcal{K}_i)|\ge \sum_{i=1}^r k_i\ge m+r-1$. Thus, $\mathcal{N}(\mathcal{T}_n,S)$ is $m$-semidense.
\end{proof}

\subsection{Counting dense subcomplexes}\label{SecDense}

\begin{lemma}\label{subcomplexcount}
Let $V$ be a $k$-element vertex set, and let $m$ be a positive integer. Then there are at most $4^{4m}3^{3m}k^{m+2}$ 
 labeled $2$-dimensional complexes $\mathcal{K}$ on the vertex set $V$ such that $|F_2(\mathcal{K})|=m$ and $H_{\mathcal{K}}$ is connected. 
\end{lemma}
\begin{proof}
 Let $\mathcal{K}$ be a $2$-dimensional complex on the vertex set $V$ such that $|F_2(\mathcal{K})|=m$ and $H_{\mathcal{K}}$ is connected. Clearly, $|F_1(\mathcal{K})|\le 3|F_2(\mathcal{K})|=3m$. Since $H_{\mathcal{K}}$ is connected, we can take a spanning tree of $H_{\mathcal{K}}$. By choosing an element of $F_1(\mathcal{K})$ as a root, we can turn this into a rooted spanning tree on at most $4m$ vertices. We will count the simplicial complexes as above based on the isomorphism class of this rooted tree. Note that the number of isomorphism classes of rooted trees on at most $4m$ vertices is at most $4^{4m}$. (Indeed, the number of isomorphism classes of rooted trees on exactly $\ell$ vertices can be bounded by the number of plane trees on $\ell$ vertices, which is given by the Catalan number $C_{\ell-1}\le 4^{\ell-1}$, see \cite[Theorem 1.5.1]{stanley2015catalan}. Thus, the number of isomorphism classes of rooted trees on at most $4m$ vertices is at most $\sum_{\ell=1}^{4m} 4^{\ell-1}<4^{4m}$.)
 
 Let $(T,o)$ be a rooted tree. Let $A$ be the set of vertices of $T$ which are at an even distance from the root $o$, and let $B$ be the set of vertices of $T$ which are at odd distance from the root $o$. Assume that $|B|=m$ and $|A|\le 3m$. Let $\mathfrak{C}$ be the set of $2$-dimensional simplicial complexes~$\mathcal{K}$ on the vertex set $V$ such that $H_{\mathcal{K}}$ has a rooted spanning tree which is isomorphic to $(T,o)$ (as a rooted graph) and the root is in $F_1(\mathcal{K})$. Let $\mathfrak{C}'$ be the set of functions $f:A\cup B\to {{V}\choose{2}}\cup {{V}\choose{3}}$ such that 
 \begin{itemize}
 \item $f(a)\in {{V}\choose{2}}$ for all $a\in A$,
 \item $f(b)\in {{V}\choose{3}}$ for all $b\in B$,
 \item $f(a)\subset f(b)$ for all $a\in A$ and $b\in B$ such that $ab$ is an edge of $T$.
 \end{itemize}
 Given any complex $\mathcal{K}\in \mathfrak{C}$ there is a surjective map $f:A\cup B\to F_1(\mathcal{K})\cup F_2(\mathcal{K})$ which induces an isomorphism between $T$ and a spanning tree of $H_\mathcal{K}$ such that $f(o)\in F_1(\mathcal{K})$. Clearly this map is in $\mathfrak{C}'$. Moreover, $\mathcal{K}$ can be recovered from $f$. Thus, $|\mathfrak{C}|\le |\mathfrak{C}'|$.

 Next, we describe a way we can generate all the elements of $\mathfrak{C}'$. Since $T$ is a rooted tree, we can define the parent of each vertex of $T$ other than the root $o$. Let $o=u_1,u_2,\dots,u_{|A|+|B|}$ be a list of all the vertices of $T$ such that for any vertex $u\neq o$ of $T$, the parent of $u$ comes before $u$. For example, one can list the vertices in the order a depth first search visits them. We define the values of $f(u_1),f(u_2),\dots,f(u_{|A|+|B|})$ one by one. For $f(u_1)=f(o)$, we have ${{k}\choose{2}}$ choices. Assume that for $i>1$, we have already chosen $f(u_1),f(u_2),\dots,f(u_{i-1})$ and we would like to choose $f(u_i)$. Let $p$ be the parent of $u_i$. Note that $f(p)$ is already chosen. If $u_i\in A$, then $f(u_i)$ must be chosen as a two-element subset of $f(p)$, so we have $3$ options. If $u_i\in B$, then $f(u_i)$ must be chosen as a $3$-element set containing $f(p)$, so we have $k-2$ options.
 Thus,
 \[|\mathfrak{C}|\le |\mathfrak{C}'|={{k}\choose{2}}(k-2)^{|B|}3^{|A|-1}\le k^{m+2}3^{3m}.\]

 Combining this with our earlier observation that the number of choices for $(T,o)$ is most~$4^{4m}$, the statement follows.\end{proof}

\begin{lemma}\label{nodensesubcomplex}
 There is a $c>0$ such that
 \begin{align*}\lim_{n\to\infty}\mathbb{P}(\mathcal{T}_n\text{ has a $k$-dense subcomplex for some }k<c\log(n))&=0,\\\lim_{n\to\infty}\mathbb{P}(\mathcal{T}_n\text{ has a $k$-semidense subcomplex for some }k<c\log(n))&=0.
 \end{align*}
\end{lemma}

\begin{proof}
 Let $V$ be a $k$-element subset of $[n]$, and let $D_V$ be the set of complexes $\mathcal{C}$ on the vertex set $V$ such that $\mathcal{C}$ has $k+1$ triangular faces and $H_\mathcal{C}$ is connected. Then combining \eqref{hadamard} and Lemma~\ref{subcomplexcount}, we obtain
 \[\sum_{\mathcal{C}\in D_V}\mathbb{P}(\mathcal{C}\subset \mathcal{T}_n)\le |D_V| \left(\frac{3}{n}\right)^{k+1}=\left(\frac{3}{n}\right)^{k+1}4^{4(k+1)}3^{3(k+1)}k^{k+3}.\]
 Combining this with the estimate ${{n}\choose{k}}\le \left(\frac{en}{k}\right)^k$, we see that
 \[\sum_{V\in {{[n]}\choose{k}}}\sum_{\mathcal{C}\in D_V}\mathbb{P}(\mathcal{C}\subset \mathcal{T}_n)\le {{n}\choose {k}}\left(\frac{3}{n}\right)^{k+1}4^{4(k+1)}3^{3(k+1)}k^{k+3}\le\frac{k^3}n e^k4^{4(k+1)}3^{4(k+1)}\le \frac{k^3}n C^k\]
 for a large enough constant $C$. Thus,
 \[\sum_{\substack{V\subset [n]\\3\le |V|\le b}}\sum_{\mathcal{C}\in D_V}\mathbb{P}(\mathcal{C}\subset \mathcal{T}_n)\le \frac{b^4}n C^b.\]
 
 If $\mathcal{C}$ is a $k$-dense complex, then there is a subcomplex $\mathcal{C}'$ of $\mathcal{C}$ with $k+1$ triangular faces such that $H_{\mathcal{C}'}$ is connected. Indeed, one can run a breadth first search on $H_{\mathcal{C}}$ and only keep the first $k+1$ triangular faces that we reach. Thus, the probability of the event that $\mathcal{T}_n$ contains a $k$-dense subcomplex for some $k<c\log(n)$ is at most
 \[\sum_{\substack{V\subset [n]\\3\le|V|\le c\log(n)}}\sum_{\mathcal{C}\in D_V}\mathbb{P}(\mathcal{C}\subset \mathcal{T}_n)\le \frac{(c\log(n))^4}n C^{c\log(n)}.\]
 Here the right hand side converges to $0$ provided that $c$ is small enough.

 Let $V$ be a $k$-element subset of $[n]$, $r>1$ and let $k_1,k_2,\dots,k_r$ be nonnegative integers such that $\sum_{i=1}^r k_i=k+r-1$. Let $D_V(k_1,\dots,k_r)$ be the set of complexes $\mathcal{C}$ on the vertex set $V$ such that $\mathcal{C}=\cup_{i=1}^r\mathcal{C}_i$, where $\mathcal{C}_i$ has $k_i$ triangular faces and $H_{\mathcal{C}_i}$ is connected, moreover, the sets $F_2(\mathcal{C}_i)$ are pairwise disjoint. Then by Lemma~\ref{subcomplexcount}, we have
 \[|D_V(k_1,\dots,k_r)|\le \prod_{i=1}^r 4^{4k_i}3^{3k_i}k^{k_i+2}=4^{4(k+r-1)}3^{3(k+r-1)}k^{k+3r-1}.\]

Combining this with \eqref{hadamard}, we obtain
\[\sum_{\mathcal{C}\in D_V(k_1,\dots,k_r)}\mathbb{P}(\mathcal{C}\subset \mathcal{T}_n)\le |D_V(k_1,\dots,k_2)| \left(\frac{3}{n}\right)^{k+r-1}=\left(\frac{3}{n}\right)^{k+r-1}4^{4(k+r-1)}3^{3(k+r-1)}k^{k+3r-1}.\]
 
 Therefore,
 \[\sum_{V\in {{[n]}\choose{k}}}\sum_{\mathcal{C}\in D_V(k_1,\dots,k_r)}\mathbb{P}(\mathcal{C}\subset \mathcal{T}_n)\le\left(\frac{en}k\right)^k \left(\frac{3}{n}\right)^{k+r-1}4^{4(k+r-1)}3^{3(k+r-1)}k^{k+3r-1}\le \frac{k^{3r-1}}{n^{r-1}} C^{k+r}\]
 for a large enough constant $C$. Let $D_V[r]$ be the union of all the sets $D_V(k_1,\dots,k_r)$, where $r>1$ and $\sum_{i=1}^r k_i=k+r-1$. Then
 \[\sum_{V\in {{[n]}\choose{k}}}\sum_{\mathcal{C}\in D_V[r]}\mathbb{P}(\mathcal{C}\subset \mathcal{T}_n)\le \frac{(k+r)^rk^{3r-1}}{n^{r-1}}C^{k+r}.\]
 
 If $D_V'$ is the union of all the sets $D_V[r]$, where ${{|V|}\choose {3}}\ge r>1$, then 
 \[\sum_{\substack{V\subset [n]\\3\le|V|\le c\log(n)}}\sum_{\mathcal{C}\in D_V'}\mathbb{P}(\mathcal{C}\subset \mathcal{T}_n)\le \sum_{r=2}^{(c \log(n))^3}\frac{(c\log(n)+r)^r(c\log(n))^{3r}}{n^{r-1}}C^{c\log(n)+r}.\]
 Here the right hand side converges to $0$ as $n\to\infty$ provided that $c$ is small enough.

 The same argument as before show that any $k$-semidense complex on the vertex set $V$, has a subcomplex which is in $D_V'$. Thus, the statement follows. 
\end{proof}

\textbf{Finishing the proof of Theorem~\ref{thm1}:}

Part~\eqref{thm1parta} of Theorem~\ref{thm1} follows by combining Lemma~\ref{oddpdense} and Lemma~\ref{nodensesubcomplex}.

Part~\eqref{thm1partb} of Theorem~\ref{thm1} follows by combining Lemma~\ref{p2dense} and Lemma~\ref{nodensesubcomplex}.

\medskip

\textbf{Open question:} 

Can we improve the bound $c\log(n)$ in Theorem~\ref{thm1}?

\section{Existence results -- The proof of Theorem~\ref{thm2}}
\subsection{Setting up the second moment method}

In this section, we prove Theorem~\ref{thm2} by a second moment argument. Our proof follows \cite{meszaros20242}, but in a more general setting.

Let $G$ be a NUG. We assign a simplicial complex $\mathcal{C}_G$ to $G$ as follows. The construction depends on whether the unique cycle of $G$ is a triangle or not. We start with the case when the unique cycle of $G$ is not a triangle. Let us orient the edges of $G$ such that the edges of the unique cycle form a directed cycle, and all the edges not in the cycle are oriented towards the cycle. This way we get an oriented graph $\vec{G}$ where the outdegree of each vertex is~$1$. Given any vertex $v$, let $p(v)$ be the unique vertex such that $(v,p(v))$ is an edge of $\vec{G}$. For all the vertices $v$ of $G$, add the triangle $\{v,p(v),p(p(v))\}$ to the complex $\mathcal{C}_G$. We may choose an orientation of the face $\{v,p(v),p(p(v))\}$ such that it aligns with both of the oriented edges $(v,p(v))$ and $(p(v),p(p(v)))$. 

Next, we move on to the case when $G$ contains a triangle. Again consider an orientation $\vec{G}$ of $G$ as above. Let $v_1,v_2,v_3$ be the oriented triangle of $\vec{G}$. Let $u_i$ be a neighbor of $v_i$ not contained in the triangle. Let $V_0=\{v_1,v_2,v_3,u_1,u_2,u_3\}$. Let us consider the following $6$ triangles:
\begin{equation}\label{sixtriangles}
 \{u_1,v_1,v_2\},\{v_1,v_2,u_2\},\{u_2,v_2,v_3\},\{v_2,v_3,u_3\},\{u_3,v_3,v_1\},\{v_3,v_1,u_1\}.
\end{equation}

Moreover, for all the vertices $v\in V(G)\setminus V_0$, consider the triangle $\{v,p(v),p(p(v))\}$. Let $\mathcal{C}_G$ consist of all of these triangles.

Let $J_G$ be the matrix obtained from the $2$-dimensional boundary matrix of $\mathcal{C}_G$ by considering only the rows that correspond to edges of $G$. Thus, the rows of $J_G$ are indexed by the edges of $G$, and the columns are indexed by the triangular faces of $\mathcal{C}_G$. Observe $J_G$ is a square matrix. Also note that in the case when the unique cycle of $G$ has length at least $5$, $J_G$ is a $0-1$ matrix, because of our choices of orientations. 

We mention that the complex $\mathcal{C}_G$ depends on the orientation of $G$ or the choices of $u_1,u_2,u_3$, but we do not indicate this dependence in the notation.

\begin{lemma}\label{detJG2}
 We have $|\det(J_G)|=2$.
\end{lemma}
\begin{proof}
 We start by the case, when $G$ is a cycle of length at least $5$. Let $v_1,v_2,\dots,v_k$ be the vertices of $G$ listed as we move along the directed cycle $\vec{G}$. Let us list the edges of $G$ in the order $(v_1,v_2),(v_2,v_3),\dots,(v_{k-1},v_k),(v_k,v_1)$. Moreover, let us list the triangular faces of $\mathcal{C}_G$ in the order $\{v_1,v_2,v_3\},\{v_2,v_3,v_4\},\dots,\{v_{k-1},v_k,v_1\},\{v_k,v_1,v_2\}$. With these orderings $J_G=(J_{i,j})_{i,j=1}^n$, where
 \[J_{i,j}=\begin{cases}
 1&\text{ if $i=j$ or $i-j\equiv 1\mod{k}$,}\\
 0&\text{ otherwise.}
 \end{cases}\]
 Thus, $J_G$ is a circulant matrix, with associated polynomial $1+x$. Relying on the fact that $k$ is odd, one can easily see that $\det J_G=2$.

 Next, the case when the unique cycle of $G$ has length at least $5$ can be handled by induction. If $G$ is not a cycle, then it has a leaf $v$. Since $v$ is a leaf, $\sigma=\{v,p(v),p(p(v))\}$ is the unique triangular face of $\mathcal{C}_G$ that contains the edge $(v,p(v))$. Thus, the row of $J_G$ corresponding to $(v,p(v))$ contains exactly one non-zero entry, namely, in the column corresponding to~$\sigma$, we have a $1$. Expanding the determinant along this row, we see that $|\det J_G|=|\det J_{G'}|$, where $G'$ is obtained from $G$ by deleting the leaf $v$. Thus, the statement follows by induction.

 Finally, we consider the case when the unique cycle of $G$ is a triangle. As before, by induction we can reduce the statement to the case when $G$ has $6$ vertices. In that case $J_G$ is the matrix
 \[ \begin{blockarray}{ccccccc}
&\{u_1,v_1,v_2\}&\{v_1,v_2,u_2\}&\{u_2,v_2,v_3\}&\{v_2,v_3,u_3\}&\{u_3,v_3,v_1\}&\{v_3,v_1,u_1\}\\
\begin{block}{c(cccccc)}
 (u_1,v_1) & +1 & 0 & 0 & 0 & 0 &-1\\
 (v_1,v_2) & +1 & +1 & 0 & 0 & 0&0\\
 (u_2,v_2) & 0 & -1 & +1 & 0 & 0&0\\
 (v_2,v_3) & 0 & 0 & +1 & +1 &0&0\\
 (u_3,v_3) & 0 & 0 & 0 & -1 &+1&0\\
 (v_3,v_1) & 0 & 0 & 0 & 0 &+1&+1\\
\end{block}
\end{blockarray}\quad.
\]
A straightforward calculation shows that $|\det(J_G)|=2$.
\end{proof}

Let $G$ be a graph with connected components $G_1,G_2,\dots,G_h$. Assume that each $G_i$ is a NUG. Let $\mathcal{C}_G$ be the union of the complexes $\mathcal{C}_{G_i}$. Let $\varphi:V(G)\to [n]$ be an injective map. We say that a $2$-dimensional simplicial complex $\mathcal{K}$ on the vertex set $[n]$ is $(G,\varphi)$-nice if $\varphi(\mathcal{C}_G)$ is a subcomplex of $\mathcal{K}$, and the only triangular faces of $\mathcal{K}$ that contain an edge of $\varphi(G)$ are the ones in $\varphi(\mathcal{C}_G)$.

The proof of the following statement is straightforward, once we observe that each triangular face of $\mathcal{C}_G$ contains exactly two edges of $G$.
\begin{lemma}\label{ifGvarphinice}
 Let $\mathcal{K}$ be $(G,\varphi)$-nice complex, and let $f $ be the unique element of $C^1(\mathcal{K},\mathbb{F}_2)$ such that $\suppr(f)=\varphi(G)$. Then $f\in Z^1(\mathcal{K},\mathbb{F}_2)$. 
\end{lemma}

Let us define the random variable
\[X_n=\left|\{\varphi:V(G)\to [n]\text{ injective }:\,\mathcal{T}_n \text{ is $(G,\varphi)$-nice}\}\right|.\]

In Section~\ref{secfirstmoment} and Section~\ref{secsecondmoment}, we prove that for all large enough $n$, we have
\begin{align}
\mathbb{E}X_n&\ge e^{-8Dm}\qquad\quad\text{ and}\label{firstmomentbound}\\
\mathbb{E}X_n^2&\le (3(m+3))^m.\label{secondmomentbound}
\end{align}
Thus, by the Paley–Zygmund inequality, for all large enough $n$, we have
\[\mathbb{P}(X_n>0)\ge \frac{\left(\mathbb{E}X_n\right)^2}{\mathbb{E}X_n^2}\ge \frac{e^{-16Dm}}{(3(m+3))^m}.\]

Note that on the event that $X_n>0$, we have an $f\in Z^1(\mathcal{T}_n,\mathbb{F}_2)$ such that $\supp_r(f)$ is isomorphic to $G$ by Lemma~\ref{ifGvarphinice}. Thus, Theorem~\ref{thm2} follows.

\subsection{Estimating the first moment}\label{secfirstmoment}
Let $\varphi:V(G)\to [n]$ be a fixed injective map. Let $A={{[n-1]}\choose{2}}\setminus \varphi(E(G))$, and let $B$ be the set of all $\sigma\in {{[n]}\choose{3}}$ such that $\sigma$ does not contain any edge of $\varphi(G)$.

Relying on these definitions, we see that a $2$-dimensional simplicial complex $\mathcal{K}$ on the vertex set $[n]$ is $(G,\varphi)$-nice if and only if 
\[
F_2(\varphi(\mathcal{C}_G))\subset F_2(\mathcal{K})\subset F_2(\varphi(\mathcal{C}_G))\cup B.
\]

Recall that $m$ denotes the number of vertices of $G$, which is the same as $|F_2(\mathcal{C}_G)|$, and $h$ denotes the number of connected components of $G$.
\begin{lemma}\label{lemma8} 
Assume that $n\notin \varphi(V(G))$. Let $\mathcal{K}$ be a $2$-dimensional simplicial complex with~${n-1}\choose{2}$ triangular faces and complete $1$-skeleton. Then $\mathcal{K}$ is a $(G,\varphi)$-nice hypertree if and only if 
\[F_2(\mathcal{K})=F_2(\varphi(\mathcal{C}_G))\cup B_0,\]
where $B_0\subset B$ such that $|B_0|={{n-1}\choose 2}-m$, and
\[\det J^r_{n,2}[A,B_0]\neq 0.\]

In this case,
\[|\det J^r_{n,2}[F_2(\mathcal{K})]|=2^h\left|\det J^r_{n,2}[A,B_0]\right|. \]

\end{lemma}
\begin{proof}

Note that $\mathcal{K}$ is $(G,\varphi)$-nice if and only if
\[F_2(\mathcal{K})=F_2(\varphi(\mathcal{C}_G))\cup B_0,\]
where $B_0\subset B$ and $|B_0|={{n-1}\choose 2}-m$.

Take a complex $\mathcal{K}$ as above. Reorder the rows of $J^r_{n,2}[F_2(\mathcal{K})]$ such that the first few rows are those which correspond to the edges of $\varphi(G_1)$, then we continue with the rows that correspond to the edges of $\varphi(G_2)$, and so on up to the edges of $\varphi(G_h)$. Finally, the last few rows are indexed by $A$.

We also can reorder the columns such that the columns corresponding to $F_2(\varphi(\mathcal{C}_{G_1}))$ come first, then the columns corresponding to $F_2(\varphi(\mathcal{C}_{G_2}))$, and so on up to the columns corresponding to $F_2(\varphi(\mathcal{C}_{G_h}))$. Finally, the last few columns are indexed by $B_0$. 

If $\tau\in F_1(\varphi(G_i))$ and $\tau$ is contained in a $\sigma\in F_2(\mathcal{K})$, then we must have $\sigma\in F_2(\varphi(\mathcal{C}_{G_i}))$. 
Thus, after the above reordering of the rows and columns, we get a block lower triangular matrix, where the $h+1$ diagonal blocks are \[\left(J^r_{n,2}[F_1(\varphi(G_i)),F_2(\varphi(\mathcal{C}_{G_i}))]\right)_{i=1}^h\quad\text{ and }\quad J^r_{n,2}[A,B_0].\] Note that these matrices are square matrices.

Therefore,
\begin{align*}\left|\det J^r_{n,2}[F_2(\mathcal{K})]\right|&=\left|\det J^r_{n,2}[A,B_0]\right|\prod_{i=1}^h \left|\det J^r_{n,2}[F_1(\varphi(G_i)),F_2(\varphi(\mathcal{C}_{G_i}))]\right|\\&=\left|\det J^r_{n,2}[A,B_0]\right|\prod_{i=1}^h \left|\det J_{G_i}\right|\\&=2^h\left|\det J^r_{n,2}[A,B_0]\right|
\end{align*}
where the last equality follows from Lemma~\ref{detJG2}.

The statement follows from Lemma~\ref{Hdet}.
\end{proof}

\begin{lemma}\label{Lemma9}
 Assume that $n\notin \varphi(V(G))$. Then
\[\mathbb{P}(\mathcal{T}_n\text{ is $(G,\varphi)$-nice})=\frac{2^{2h} \det(M_G) }{n^{{n-2}\choose{2}}},\]
where $M_G=J^r_{n,2}[A,B] (J^r_{n,2}[A,B])^T$.
\end{lemma}
\begin{proof}
Combining Lemma~\ref{lemma8} with the Cauchy-Binet formula, we have
\[\sum_{\substack{\mathcal{T}\text{ is $(G,\varphi)$-nice}\\\text{hypertree on $[n]$}}} |\det J^r_{n,2}[F_2(\mathcal{T})]|^2=\sum_{\substack{B_0\subset B\\|B_0|={{n-1}\choose {2}}-m}} 2^{2h}|\det J^r_{n,2}[A,B_0]|^2=2^{2h} \det(M_G).\]
Combining this with \eqref{measuredef} and Lemma~\ref{Hdet}, the statement follows.
\end{proof}

Recall that $D$ denotes the maximum degree of $G$.

\begin{lemma}\label{phinicebound}
Assuming that $n$ is large enough, for any injective map $\varphi:V(G)\to [n]$, we have
\[\mathbb{P}(\mathcal{T}_n\text{ is $(G,\varphi)$-nice })\ge 2^{2h} n^{-m} e^{-8Dm}.\]
\end{lemma}
\begin{proof}
 By symmetry we may assume that $n\not\in \varphi(V(G))$. Let $B^r=B\cap {{[n-1]}\choose{3}}$. Note that
\[M_G-I=J^r_{n,2}[A,B^r] (J^r_{n,2}[A,B^r])^T.\]
One can consider $A$ as a graph, or in other words, a $1$-dimensional simplicial complex on the vertex set $[n-1]$. Note that since $A$ is connected, $B^1(A,\mathbb{R})$ has dimension $(n-1)-1=n-2$. Also, $B^1(A,\mathbb{R})$ is contained in the kernel of $J^r_{n,2}[A,B^r] (J^r_{n,2}[A,B^r])^T$. Thus, $1$ is an eigenvalue of $M_G$ with multiplicity at least $n-2$. 

We have
\[J^r_{n,2}(J^r_{n,2})^T-nI=-(J_{n-1,1})^T J_{n-1,1},\]
see \cite[(2) in the proof of Lemma 3]{kalai1983enumeration}.

Thus, 
\begin{equation}\label{meq1}
 J^r_{n,2}[A,*](J^r_{n,2}[A,*])^T-nI=-(J_{n-1,1}[*,A])^TJ_{n-1,1}[*,A].
\end{equation}

Let $\overline{B}={{[n]}\choose {3}}\setminus B$. Then
\begin{equation}\label{meq2}
 J^r_{n,2}[A,*](J^r_{n,2}[A,*])^T=J^r_{n,2}[A,B](J^r_{n,2}[A,B])^T+J^r_{n,2}[A,\overline{B}](J^r_{n,2}[A,\overline{B}])^T.
\end{equation}
Let us introduce the notations \begin{align*}N_G&=(J_{n-1,1}[*,A])^TJ_{n-1,1}[*,A],\\M_G'&=J^r_{n,2}[A,\overline{B}](J^r_{n,2}[A,\overline{B}])^T.\end{align*} 
Combining \eqref{meq1} and \eqref{meq2}, we obtain that
\[M_G=nI-N_G-M_G'.\]
Note that $J_{n-1,1}[*,A]$ has $n-1$ rows, which are linearly dependent since their sum is $0$. Thus, the rank of $J_{n-1,1}[\cdot,A]$ is at most $n-2$. Therefore, $N_G$ has rank at most $n-2$. Thus,
\[\dim \ker N_G\ge {{n-1}\choose{2}}-m-(n-2)={{n-2}\choose{2}}-m.\]

Let us consider the graph $H_{(G,\varphi)}$ on the vertex set $\overline{B}$, where two distinct triangular faces $\sigma_1,\sigma_2\in \overline{B}$ are connected if there is an edge $\tau\in A$ such that $\tau\subset \sigma_1$ and $\tau\subset \sigma_2$. If for an edge $\{u,v\}\in A$ and vertex $w$, we have $\{u,v,w\}\in \overline{B}$, then $\{u,w\}$ or $\{v,w\}$ is an edge of $\varphi(G)$. Thus, if $D$ is the maximum degree of $G$, then for all $\tau\in A$, there are at most $2D$ triangles in $\sigma\in\overline{B}$ such that $\tau\subset \sigma$. The boundary of any $\sigma\in\overline{B}$ contains at most $2$ edges from $A$. Thus, the degree of $\sigma$ in the $H_{(G,\varphi)}$ is at most $2(2D-1)$. Therefore, the graph $H_{(G,\varphi)}$ has a proper coloring with $\chi=2(2D-1)+1$ colors. That is, there is partition of $\overline{B}$ into sets $L_1,L_2,\dots,L_{\chi}$ such that if $\sigma_1,\sigma_2\in L_i$, and $\sigma_1\neq \sigma_2$, then $\partial \sigma_1\cap \partial \sigma_2\cap A= \emptyset$. Let \[Q_i=J^r_{n,2}[A,L_i](J^r_{n,2}[A,L_i])^T.\] Then $M_G'=\sum_{i=1}^\chi Q_i$. Moreover, $Q_i$ is block diagonal matrix, where each diagonal block is either
\[(0),\quad (1),\quad \begin{pmatrix} 1&1\\1&1\end{pmatrix},\quad \text{or}\quad \begin{pmatrix} 1&-1\\-1&1\end{pmatrix}.\]
Thus, the operator norm of $Q_i$ is at most $2$. Thus, $M_G'$ has operator norm at most $2\chi$. The rank of $M_G'$ is clearly at most $|\overline{B}|\le mn$. It follows that $\dim (\ker N_G\cap \ker M_G')\ge \dim \ker N_G-mn$. All the vectors in $\ker N_G\cap \ker M_G'$ are eigenvectors of $M_G$ with eigenvalue~$n$. Moreover, for any vector in $v\in\ker N_G$, we have $v^TM_G v\ge (n-2\chi)\|v\|_2^2$. Thus, by the Courant-Fischer theorem, we obtain the following statement: Let us consider the eigenvalues of $M_G$ in decreasing order. Then the first ${{n-2}\choose 2}-m-mn$ eigenvalues are all at least $n$, the next $mn$ eigenvalues are all at least $n-2\chi$. Finally, we also have at least $n-2$ eigenvalues equal to $1$. Assuming that $n>2\chi+1$, these are at least ${{n-1}\choose2}-m$ eigenvalues. Thus, we found all the eigenvalues of $M_G$. Thus,
\[\det M_G\ge n^{{{n-2}\choose 2}-m-mn} (n-2\chi)^{mn}. \]
Combining this Lemma~\ref{Lemma9}, we obtain that
\begin{align*}\mathbb{P}(\mathcal{T}_n\text{ is $(G,\varphi)$-nice})&\ge \frac{2^{2h} n^{{{n-2}\choose 2}-m-mn} (n-2\chi)^{mn}}{n^{{n-2}\choose {2}}}\\&=2^{2h} n^{-m}\left(1-\frac{2\chi}{n}\right)^{nm}\\&\ge 2^{2h} n^{-m} e^{-8Dm}
\end{align*}
for all large enough $n$.
\end{proof}

The number of injective maps $\varphi:V(G)\to [n]$ is
\[\prod_{i=0}^{m-1}(n-i)\ge 2^{-2h}n^m\]
for all large enough $n$. Combining this with Lemma~\ref{phinicebound}, we obtain the estimate in \eqref{firstmomentbound}.
\subsection{Estimating the second moment}\label{secsecondmoment}

\begin{lemma}\label{VFF}
 Let $G$ be a NUG, and let $\mathcal{C}_G$ be the associated complex. Let $F\subset F_2(\mathcal{C}_G)$, and let $V_F=\cup_{\sigma\in F} \sigma$. Then $|V_F|\ge |F|$.
\end{lemma}
\begin{proof}

First, assume that $G$ the unique cycle of $G$ has length at least $5$. In that case, we can prove the statement by defining an injective map $\alpha:F\to V_F$ as follows. Given $\sigma\in F$, $\sigma$ is of the form $\{u,p(u),p(p(u))\}$, we set $\alpha(\sigma)=u$.

Next, consider the case when the unique cycle of $G$ is a triangle. Let $v_1,v_2,v_3,u_1,u_2,u_3$ be the six vertices that we used in the definition $\mathcal{C}_G$. Let $F_1$ be the set of the $6$ triangles listed in~\eqref{sixtriangles}. We again define an injective map $\alpha:F\to V_F$, but the definition is slightly more complicated in this case. First, we define $\alpha(\sigma)$ for $\sigma\in F\cap F_1$ as follows:
\[
\alpha(\sigma)=\begin{cases}
 u_i&\text{ if }\sigma=\{u_i,v_i,v_{i+1}\},\\
 v_i&\text{ if }\sigma=\{v_{i-1},v_i,u_i\}.
\end{cases}
\]
Note that indices are modulo $3$ in the definition above. If $\sigma\in F\setminus F_1$, then $\sigma$ is of the form $\{u,p(u),p(p(u))\}$ for some $u$ distinct from $v_1,v_2,v_3,u_1,u_2,u_3$. We set $\alpha(\sigma)=u$. It is straightforward to check that $\alpha$ is again injective.
\end{proof}

Recall that we have set $m=|V(G)|$. For $0\le k\le m$, we define
\begin{multline*}\Phi_{G,k,n}=\{(\varphi_1,\varphi_2)\,:\, \text{$\varphi_1$ and $\varphi_2$ are injective maps from $V(G)$ to $[n]$,} \\\text{ and }|F_2(\varphi_1(\mathcal{C}_{G}))\cap F_2(\varphi_2(\mathcal{C}_{G}))|=k \}.\end{multline*}

Given an injective map $\varphi_1:V(G_1)\to [n]$, we would like to estimate the number of maps $\varphi_2$ such that $(\varphi_1,\varphi_2)\in \Phi_{G,k,n}$. If $|F_2(\varphi_1(\mathcal{C}_{G}))\cap F_2(\varphi_2(\mathcal{C}_{G}))|=k$, then $|\varphi_1(V(G))\cap \varphi_2(V(G))|\ge k$ by Lemma~\ref{VFF}. Thus, we have at least $k$ vertices $v\in V(G)$ such that $\varphi_2(v)\in \varphi_1(V(G))$. Thus, we have at most ${{m}\choose {k}} m^k n^{m-k}$ choices for $\varphi_2$ given $\varphi_1$. The number of choices for $\varphi_1$ is at most~$n^m$. Therefore,
\[|\Phi_{G,k,n}|\le {{m}\choose {k}} m^k n^{2m-k}.\]

Observe that
\begin{align}
\mathbb{E}X_n^2&=\sum_{k=0}^m\sum_{(\varphi_1,\varphi_2)\in \Phi_{G,k,n}} \mathbb{P}(\mathcal{T}_n\text{ is $(G,\varphi_1)$-nice and $(G,\varphi_2)$-nice})\nonumber\\
&\le \sum_{k=0}^m\sum_{(\varphi_1,\varphi_2)\in \Phi_{G,k,n}} \mathbb{P}(F_2(\varphi_1(\mathcal{C}_G))\cup F_2(\varphi_2(\mathcal{C}_G)) \subset\mathcal{T}_n).\label{tobec}
\end{align}
Note that $(\varphi_1,\varphi_2)\in \Phi_{G,k,n}$, we have $|F_2(\varphi_1(\mathcal{C}_G))\cup F_2(\varphi_2(\mathcal{C}_G))|=2m-k$. Combining this with~\eqref{hadamard}, we see that for all $(\varphi_1,\varphi_2)\in \Phi_{G,k,n}$, we have
\[\mathbb{P}(F_2(\varphi_1(\mathcal{C}_G))\cup F_2(\varphi_2(\mathcal{C}_G)) \subset\mathcal{T}_n)\le \left(\frac{3}n\right)^{2m-k}.\]

Therefore, continuing \eqref{tobec},

\[
\mathbb{E}X_n^2\le \sum_{k=0}^m |\Phi_{G,k,n}| \left(\frac{3}n\right)^{2m-k}\le \sum_{k=0}^m {{m}\choose {k}} m^k 3^{2m-k}=(3(m+3))^m. 
\]
obtaining the estimate in \eqref{secondmomentbound} on the second moment of $X_n$.

\bibliography{references}

\bibliographystyle{plain}

\bigskip

\noindent Andr\'as M\'esz\'aros, \\
HUN-REN Alfr\'ed R\'enyi Institute of Mathematics, \\Budapest, Hungary,\\ {\tt meszaros@renyi.hu}

\end{document}